\documentclass[11pt]{amsart}

\usepackage{latexsym}
\usepackage{color}

\usepackage{tikz-cd}
\usepackage[all]{xy}
\usepackage[utf8x]{inputenc} 
\usepackage[english]{babel}
\usepackage{mathtools}
\usepackage[T1]{fontenc}
 \usepackage{lmodern}
\usepackage{mathrsfs}
\usepackage{enumitem}
\usepackage{commath}
\usepackage{caption}

\usepackage{graphicx,epsfig,amscd,graphics,subcaption}
\usepackage[all]{xy}
\usepackage{amsmath,amssymb,amsthm}

\usepackage[margin=4cm]{geometry}

\usepackage{hyperref}
\hypersetup{
    bookmarks=true,         
    unicode=false,          
    pdftoolbar=true,        
    pdfmenubar=true,        
    pdffitwindow=false,     
    pdfstartview={FitH},    
    pdftitle={Parabolic degrees and Lyapunov exponents for hypergeometric local systems},    
    pdfauthor={Charles Fougeron},     
    pdfsubject={Subject},   
    pdfcreator={Creator},   
    pdfproducer={Producer}, 
    pdfkeywords={keyword1} {key2} {key3}, 
    pdfnewwindow=true,      
    colorlinks=false,       
    linkcolor=red,          
    citecolor=green,        
    filecolor=magenta,      
    urlcolor=cyan           
}

\newcommand{\C}{\mathbb{C}}
\newcommand{\Z}{\mathbb{Z}}
\newcommand{\R}{\mathbb{R}}

\newcommand{\Hy}{\mathbb{H}}

\newcommand{\SL}{\mathrm{SL}}

\newcommand{\PSLZ}{\operatorname{PSL}\left( \mathbb Z \right)}
\newcommand{\cpone}{\mathbb P^1}
\newcommand{\cpminus}{\cpone-\{0,1,\infty\}}
\newcommand{\D}{\operatorname{D}}
\newcommand{\dz}{\od {} {z}}

\newcommand{\bigslant}[2]{{\raisebox{.2em}{$#1$}\left/\raisebox{-.2em}{$#2$}\right.}}

\newcommand{\E}[1]{e^{2i\pi#1}}
\newcommand{\Em}[1]{e^{-2i\pi#1}}
\newcommand{\Es}[1]{E \left( #1 \right)}

\newcommand{\Dx}{\mathcal D_{\overline{\mathcal C}}}

\newcommand{\gr}{\operatorname{gr}}
\newcommand{\holb}{\mathcal E}

\newcommand{\Id}{\operatorname{Id}}

\newtheorem{theorem}{Theorem}[section]
\newtheorem*{theorem*}{Theorem}
\newtheorem{lemma}[theorem]{Lemma}
\newtheorem*{lemma*}{Lemma}
\newtheorem{proposition}[theorem]{Proposition}
\newtheorem{corollary}[theorem]{Corollary}
\newtheorem*{corollary*}{Corollary}
\newtheorem{definition}[theorem]{Definition}
\newtheorem*{remark}{Remark}

\title{Parabolic degrees and Lyapunov exponents for hypergeometric local systems}
\author{Charles Fougeron}

\begin{document}
\maketitle

\begin{abstract} 
  Consider the flat bundle on $\cpminus$ corresponding to solutions of
  the hypergeometric differential equation
  $$ \prod_{i=1}^h (\D - \alpha_i) - z \prod_{j=1}^h (\D - \beta_j) = 0, \; \text{where} \ \D = z \dz $$
  For $\alpha_i$ and $\beta_j$ real numbers, this bundle is known to
  underlie a complex polarized variation of Hodge structure.  Setting
  the complete hyperbolic metric on $\cpminus$, we associate $n$
  Lyapunov exponents to this bundle.  We compute the parabolic degrees
  of the holomorphic subbundles induced by the variation of Hodge
  structure and study the dependence of the Lyapunov exponents in
  terms of these degrees by means of numerical simulations.
\end{abstract}

\section{Introduction}

Oseledets decomposition of flat bundles over an ergodic dynamical
system is often referred to as \textit{dynamical variation of Hodge
  structure}.  In the case of Teichmüller dynamics, both Oseledets
decomposition and a variation of Hodge structure (VHS) appear. Two
decades ago it was observed in \cite{Kontsevich} that these structures
were linked, their invariants are related: the sum of the Lyapunov
exponents associated to a Teichmüller curve equals the normalized
degree of the Hodge bundle.  This formula was studied extensively and
extended to strata of abelian and quadratic differentials from then
(see \cite{Forni}, \cite{Krikorian}, \cite{BouwMoller}, \cite{EKZ}).
Soon this link was observed in other settings.  In \cite{KappesMoller}
it was used as a new invariant to classify hyperbolic structures and
distinguish Deligne-Mostow's non-arithmetic lattices in $\SL_2(\C)$.
In \cite{Filip} a similar formula was observed for higher weight
variation of Hodge structures.  The leitmotiv in this work is the
study of the relationship between theses two structures in a broad
class of examples with arbitrary weight.  This family of examples will
be given by hyperelliptic differential equations which yield a flat
bundle endowed with a variation of Hodge structure over the sphere
with three punctures.  A recent article \cite{EKMZ} shows that the
degrees of holomorphic flags of the Hodge filtration bound by below
the sum of Lyapunov exponents.  Our investigation will start by
computing these degrees and then explore the behaviour of Lyapunov
exponents through numerical simulations and their distance to the
latter lower bounds. This will enable us to bring out some simple
algebraic
relations under which there is a conjectural equality.\\


\paragraph{\bf{Hypergeometric equations.}}
Let $\alpha_1, \alpha_2, \dots, \alpha_n$ and
$\beta_1, \beta_2, \dots, \beta_n$ be two disjoint sequences of $n$
real numbers.  We define the \textit{hypergeometric differential
  equation} corresponding to those parameters

\begin{equation}
  \label{hypergeometric}
  \prod_{i=1}^n (\D - \alpha_i) - z \prod_{j=1}^n (\D - \beta_j) = 0, \; \text{where} \ \D = z \dz
\end{equation}

\noindent This equation originates from a large class of special
functions called \textit{generalized hypergeometric functions} which
satisfies
it. For more details about these functions see for example \cite{MyLove}.\\

\noindent It is an order $n$ differential equation with three
singularities at 0, 1 and $\infty$ hence the space of solutions is
locally a dimension $n$ vector space away from singularities and can
be seen in a geometrical way as a flat bundle over $\cpminus$.  This
flat bundle is completely described by its monodromy matrices around
singularities.  We will be denoting monodromies associated to simple
closed loop going counterclockwise around $0, 1$ and $\infty$ by
$M_0, M_1$ and $M_\infty$. We get a first relation between these
matrices observing that composing the three loops in the same order
will give a trivial loop: $M_\infty M_1 M_0 = \Id$.  The eigenvalues
of $M_0$ and $M_\infty$ can be expressed with parameters of the
\textit{hypergeometric equation} (\ref{hypergeometric}) and $M_1$ has
a very specific form as stated in the following proposition.

\begin{proposition}
  \label{eigenvalues}
  For any two sequences of real numbers $\alpha_1, \dots, \alpha_n$
  and $\beta_1, \dots, \beta_n$,
  \begin{itemize}
  \item $M_0$ has eigenvalues $\E{\alpha_1}, \dots, \E{\alpha_n}$
  \item $M_\infty$ has eigenvalues $\Em{\beta_1}, \dots, \Em{\beta_n}$
  \item $M_1$ is the identity plus a matrix of rank one
  \end{itemize}
\end{proposition}

\begin{proof}
  See Proposition 2.1 in \cite{Fedorov} or alternatively Prop. 3.2 and
  Theorem 3.5 in \cite{BeukersHeckman}
\end{proof}

\noindent This proposition determines the conjugacy class of the
representation associated to the flat bundle
$\pi_1 \left( {\cpminus} \right) \rightarrow GL_n \left(\mathbb R
\right)$ thanks to the rigidity of hypergeometric equations (see
\cite{BeukersHeckman}).
They will be computed explicitly in section \ref{secmonodromy}.\\

\paragraph{\bf{Lyapunov exponents.}}
We now endow the 3 punctured sphere with its hyperbolic metric.  As
this metric implies an ergodic geodesic flow, for any integrable norm
on the flat bundle we associate to it , using Oseledets theorem, a
measurable flag decomposition of the vector bundle and $n$ Lyapunov
exponents.  These exponents correspond to the growth of the norm of a
generic vector in each flag while transporting it along with the flat
connection.\\

\noindent According to \cite{EKMZ} there is a canonical family of
integrable norms on the flat bundle associated to the hypergeometric
equation which will produce the same flag decomposition and Lyapunov
exponents.  This family contains the harmonic norm induced by the VHS
structure
and the norm we will use in our algorithm.\\

\paragraph{\bf{Variation of Hodge structure.}}
Hypergeometric equations on the sphere are well known to be physically
rigid (see \cite{BeukersHeckman} or \cite{Katz}) and this rigidity
together with irreducibility is enough to endow the flat bundle with a
VHS using its associated Higgs bundle structure (see \cite{Fedorov} or
directly Cor 8.2 in \cite{Simpsona}).  Using techniques from
\cite{Katz} and \cite{Sabbah}, Fedorov gives in \cite{Fedorov} an
explicit way to compute the Hodge numbers for the underlying VHS. We
extend this computation and give a combinatorial point of view that
will be more convenient in the following to express
parabolic degrees of the Hodge flag decomposition.\\

\noindent We introduce a canonical way to describe combinatorics of
the intertwining of $\alpha$'s and $\beta$'s on the circle $\R /
\Z$. Starting from any eigenvalue, we browse the circle
counterclockwise (or in the increasing direction for $\R$) and denote
$\alpha_1, \dots, \alpha_n, \beta_1, \dots, \beta_n$ by order of
appearance $\eta_1, \eta_2, \dots, \eta_{2n}$ and define
$\tilde f : \Z \cap [0, 2n] \mapsto \Z$ recursively by the following
properties,
\begin{itemize}
\item $\tilde f(0)=0$
\item $\tilde f(k) = \tilde f(k-1) + \left\{
    \begin{array}{rl}
      1& \text{ if } \eta_k \text{ is an } \alpha \\
      -1& \text{ if } \eta_k \text{ is an } \beta \\
    \end{array}
  \right.  $
\end{itemize}
\vspace{.3cm}

\noindent Let $f$ be defined on the eigenvalues by
$f(\eta_k) = \tilde f(k)$.  It depends on the choice of starting point
up to a shift. For a canonical definition, we shift $f$ such that its
minimal value is $0$. It is equivalent to starting at the point of
minimal value. This defines a unique $f$ which we \textit{intertwining
  diagram} of the equation.

\begin{figure}[h]
  \begin{subfigure}[c]{.45\linewidth}
    \centering
    \includegraphics[height=6cm]{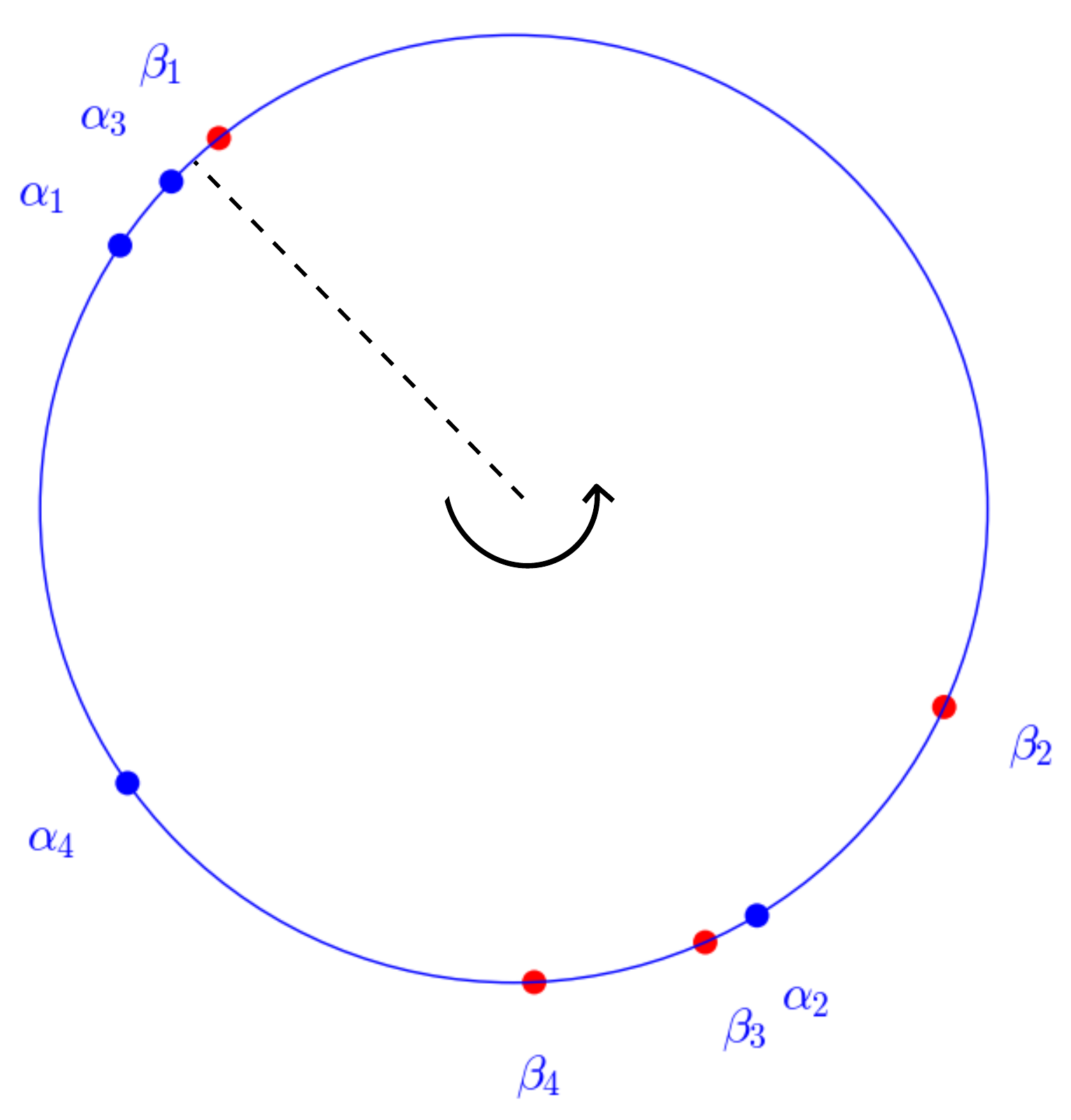}
  \end{subfigure}
  \begin{subfigure}[c]{.5\linewidth}
    \centering
    \includegraphics[height=6cm]{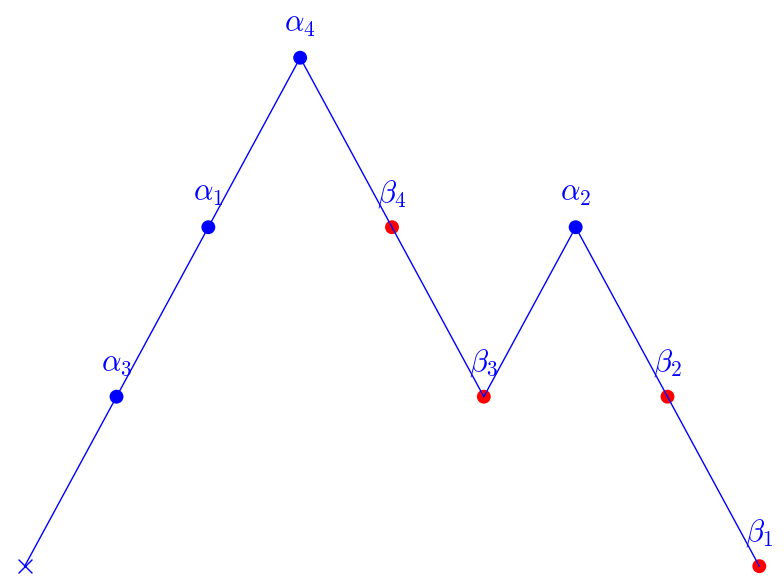}
  \end{subfigure}
  \caption{Example of computation of $f$}
\end{figure}

\vspace*{.5cm}

\noindent For every integer $1 \leq i \leq n$ we define
$$h_i := \# \{ \alpha \ | \ f(\alpha)=i \} \ = \ \# \{ \beta \ | \ f(\beta) = i-1 \}$$
Then we have the following theorem,

\begin{theorem*}[Fedorov]
  The $h_1, h_2, \dots, h_{n}$ are the Hodge numbers of the VHS after
  an appropriate shifting.
\end{theorem*}

\begin{remark}
  If the $\alpha$'s and $\beta$'s appear in an alternate order then
  $f(\alpha) \equiv 1$ and $f(\beta) \equiv 0$ thus their is just one
  element in the Hodge decomposition and the polarization form is
  positive definite.  In other words the harmonic norm is invariant by
  the flat connection.  This implies that Lyapunov exponents are zero.\\

  \noindent In general, this Hodge structure endows the flat bundle
  with a pseudo-Hermitian form of signature $(p,q)$ where $p$ is the
  sum of the even Hodge numbers and $q$ the sum of the odd ones. This
  gives classically the fact that the Lyapunov spectrum is symmetric
  with respect to $0$ and that at least $|p-q|$ exponents are zero
  (see Appendix A in \cite{Forni}).
\end{remark}

Pushing further methods of \cite{Fedorov} and \cite{Sabbah}, we
compute the parabolic degree of the sub Hodge bundles.  This
computation was done with the help of computer experiments in section
\ref{n=2} which yielded a conjectural formula for these degrees.
Besides from the intertwining diagram, another quantity emerges to
express them; relabel $\alpha$ and $\beta$ by order of appearance
after choosing $\alpha_1$ such that $f(\alpha_1) = 0$, then take the
representatives of $\alpha$ and $\beta$ in $\R$ which are included in
$[\alpha_1, \alpha_1 +1 [$ and define
$\gamma := \sum \beta - \sum \alpha$.  The formula will depend on the
floor value of $\gamma$. As $0 < \gamma < n$ we have $n$ possible
values $0 \leq [\gamma] < n$.

\begin{theorem}
  Let $1 \leq p \leq n$ and $\mathcal E^p$ the $p$-th graded piece of
  the Hodge filtration on $\cpminus$. We denote by $\delta^p$ the
  degree of the Deligne compactification of $\mathcal E^p$ on the
  sphere. Then,

  \begin{itemize}
  \item if $p = [\gamma] + 1$
    $$\deg_{par}(\mathcal E^p) = \delta^p + \{ \gamma \} + \sum_{f(\alpha)=p} \alpha
    + \sum_{f(\beta)=p-1} 1 - \beta $$
  \item otherwise
    $$\deg_{par}(\mathcal E^p) = \delta^p + \sum_{f(\alpha)=p} \alpha
    + \sum_{f(\beta)=p-1} 1 - \beta $$
  \item
    $$-\delta^p(V) = \# \left\{ \beta_i \ | \ f(\beta_i) = p-1 \text{ and } i \leq n-[\gamma] \right\}$$
  \end{itemize}
\end{theorem}

\subsection*{Acknowledgement}
I am very grateful to Maxim Kontsevich for sharing this problem and
taking time to discuss it. I thank dearly Jeremy Daniel for his
curiosity to the subject and his answer to my myriad of questions as
well as Bertrand Deroin; Anton Zorich for his flawless support and
attention, Martin Möller and Roman Fedorov for taking time to explain
their understanding of the parabolic degrees and Hodge invariants at
MPIM in Bonn. I am also very thankful to Carlos Simpson for his kind
answers and encouragements.

\section{Degree of Hodge subbundles}

\subsection{Variation of Hodge Structure}
\label{secvhs}
We start recalling the definition of complex variations of Hodge structures (VHS).\\

A ($\C$-)VHS on a curve $C$ consists of a complex local system
$\mathbb V_\C$ with a connection $\nabla$ and a decomposition of the
Deligne extension
$\mathcal V = \bigoplus_{p \in \mathbb Z} \mathcal E^p$ into
$C^\infty$-subbundles, satisfying:

\begin{itemize}
\item $\mathcal F^p := \bigoplus_{i \geq p} \mathcal E^i$ (resp.
  $\overline{\mathcal F^p} := \bigoplus_{i \leq p} \mathcal E^i$) are
  holomorphic (resp. antiholomorphic) subbundles for every $p \in \Z$.
\item The connection shifts the grading by at most one, i.e.
  $$\nabla(\mathcal F^p) \subset \mathcal F^{p-1} \otimes \Omega_{\mathcal C}^1
  \text{ and } \nabla(\overline{\mathcal F^p}) \subset
  \overline{\mathcal F^{p-1}} \otimes \Omega_{\mathcal C}^1 $$
\end{itemize}

Up to a shift, we can assume that there is a $n$ such that
$\mathcal E^i = 0$ for $i < 0$ and $i>n$.  We call $n$ the weight of
the VHS.  We can introduce for convenience the notation
$\mathcal E^{p, n-p} := \mathcal F^p / \mathcal F^{p-1}$.  Then we
have a $C^\infty$ isomorphism between the bundles:
$$ \bigoplus_{i=0}^n \mathcal E^{i, n-i} \cong \mathbb V$$

\subsection{Decomposition of an extended holomorphic bundle}

Let $\mathcal C$ be a complex curve, we assume that its boundary set
$\Delta := \overline{\mathcal C} \backslash \mathcal C$ is an union of
points. Consider $\holb$ an holomorphic bundle on
$\overline{\mathcal C}$. We introduce structures which will appear on
such holomorphic bundle when they are obtained by canonical extension
when we compactify $\mathcal C$.  The first one will take the form of
filtrations on each fibers above points of $\Delta$.

\begin{definition}[Filtration]
  A $[0,1)$-filtration on a complex vector bundle $V$ is a collection
  of real weights $0 \leq w_1 < w_2 < \dots < w_n < w_{n+1} = 1$ for
  some $n \geq 1$ together with a filtration of sub-vector spaces
$$G^\bullet : V = V^{\geq w_1} \supsetneq V^{\geq w_2} \supsetneq \dots \supsetneq V^{\geq w_{n+1}} = V^{\geq 1} = 0$$
The filtration satisfies $V^{\geq \nu} \subset V^{\geq \omega}$
whenever $\nu \geq \omega$ and the previous weights satisfy
$V^{\geq w_i + \epsilon} \subsetneq V^{\geq w_i}$ for any
$\epsilon > 0$.\\

We denote the graded vector bundles by
$\gr_{w_i} := \bigslant{V^{\geq w_i}}{V^{\geq w_i + \epsilon}}$ for
$\epsilon$ small.  The degree of such a filtration is by definition
$$\deg(G^\bullet) := \sum_{i=1}^{n} w_i \dim(\gr_{w_i})$$
\end{definition}

This leads to the next definition,

\begin{definition}[Parabolic structure]
  
  A \textit{parabolic structure} on $\holb$ with respect to $\Delta$
  is a couple $(\holb, G^\bullet)$ where $G^\bullet$ defines a
  $[0,1)$-filtration $G^\bullet \holb_s$ on every fiber $\holb_s$ for
  any $s \in \Delta$.\\
  A \textit{parabolic bundle} is a holomorphic bundle endowed with a
  parabolic structure.\\
  The \textit{parabolic degree} of $(\holb, G^\bullet)$ is defined to
  be
$$\deg_{par}(\holb, G^\bullet) := \deg(\holb) + \sum_{s\in\Delta} \deg(G^\bullet \holb_s)$$
\end{definition}

\subsection{Deligne extension}
In the following we consider $\mathbb V$ a flat bundle on $\cpminus$
associated to a monodromy representation with norm one eigenvalues.
We denote by $\mathcal V_{\mathcal C}$ the associated holomorphic
vector bundle.

We recall the construction of Deligne's extension of
$\mathcal V_{\mathcal C}$ which defines a holomorphic bundle on
$\overline{\mathcal C}$ with a logarithmic flat connection. We
describe it on a small pointed disk centered at $s \in \Delta$ with
coordinate $q \in D^*$. Let $\rho$ be a ray going outward of the
singularity, then we can speak of flat sections along the ray
$L(\rho)$ which has the same rank $r$ as $\mathcal V$.  As all the
$L(\rho)$ are isomorphic, we choose to denote it by $V^0$.  There is a
monodromy transformation $T: V^0 \rightarrow V^0$ to itself obtained
after continuing the solutions. This corresponds to the monodromy
matrix in the given representation. For every $\alpha \in [0,1)$ we
define
$$W_\alpha = \{ v \in V^0: (T-\zeta_\alpha)^rv= 0\}\text{ where } \zeta_\alpha = e^{2i\pi\alpha}$$

These vector spaces are non trivial for finitely many
$\alpha_i \in [0,1)$.  We define
$$T_\alpha = \zeta_\alpha^{-1} T_{|W_\alpha} \text{ and } N_\alpha = \log T_\alpha$$

Let $q: \mathbb H \to D^*, q(z) = e^{2i\pi z}$ be the universal cover
of $D^*$. Choose a basis $v_1, \dots, v_r$ of $V^0$ adapted to the
generalized eigenspace decomposition
$V^0 = \bigoplus_\alpha W_\alpha$. We consider $v_i(z)$ as the pull
back of $v_i$ on $\mathbb H$. If $v_i \in W_\alpha$, then we define
$$\widetilde{v}_i(z) = \exp(2i\pi\alpha z + z N_\alpha) v_i$$

These sections are equivariant under $z \mapsto z+1$ hence they give
global sections of $\mathcal V_{\mathcal C}(D^*)$. The Deligne
extension of $\mathcal V_{\mathcal C}$ is the vector bundle whose
space of section over $D$ is the $\mathcal O_{\mathcal D}$-module
spanned by $\tilde v_1, \dots, \tilde v_r$.
This construction naturally gives a filtration on $V^0$.\\

In general, we can define various extensions
$V^a \subset V^{-\infty} \subset j_* V$ where $j$ is the inclusion
$j: \mathcal C \rightarrow \overline{\mathcal C}$, $V^\infty$ is the
Deligne's meromorphic extension and $V^a$ (resp. $V^{>a}$) for
$a \in \R$ is the free $\mathcal O_{\overline{\mathcal C}}$-module on
which the residue of $\nabla$ has eigenvalues $\alpha$ in $[a, a+1)$
(resp.  $(a, a+1]$).  The bundle $V^\bullet$ is a filtered vector
bundle in the definition
of \cite{EKMZ}.\\

\noindent If we have a VHS $F^\bullet$ on $\holb$ over $\mathcal C$,
it induces a filtration of every $V^a$ simply by taking
$$F^p V^a := j_* F^p V \cap V^a$$
this is a well defined vector bundle thanks to Nilpotent orbit theorem
(see \cite{Sabbah}).\\

\noindent We define over some singularity $s\in\Delta$, for
$a \in (-1, 0]$ and $\lambda = \exp(- 2i\pi a)$,
$$\psi_\lambda (V^{-\infty}) = \gr_V^a = \bigslant{V^a}{V^{>a}}$$

\begin{definition}[Local Hodge data]
  For $a \in [0,1)$, $\lambda = \exp(2i\pi a)$, $ p \in \Z$ and
  $l \in \mathbb N$, we set for any $s \in \Delta$
  
  \begin{itemize}
  \item $\nu_{\alpha}^p = \dim \gr_F^p \psi_\lambda(V_s)$ also written
    $h^p \psi_\lambda(V_s)$
  \item $h^p(V) = \sum_\alpha \nu_\alpha^p(V_s)$
  \end{itemize}
\end{definition}

Simpson's theory (\cite{Simpsona}) claims that for any local system
with all eigenvalues of the form $exp(-2\pi i \alpha)$ at the
singularities endowed with a trivial filtration we associate a
filtered $\Dx$-module with residues and jumps both equal to $\alpha$.
Thus the sub $\Dx$-module corresponding to the residue $\alpha$ has
only one jump of full dimension at $\alpha$, and
\begin{equation}
  \label{parabolic degree}
  \deg_{par} (\gr_F^p V) = \delta^p(V) + \sum_{s \in \Delta,\alpha} \alpha \nu_\alpha^p(V_s)
\end{equation}
where we choose $\alpha \in [0,1)$. \\

\subsection{Acceptable metrics and metric extensions}
The above Deligne extension has a geometric interpretation when we
endow $\mathcal C$ with a acceptable metric $K$.  If $V$ is a
holomorphic bundle on $\mathcal C$, we define the sheaf
$\Xi(V)_\alpha$ on $\mathcal C \cup \{s\}$ as follows.  The germs of
sections of $\Xi(E)_\alpha$ at $s$ are the sections $s(q)$ in $j_* V$
in the neighborhood of $s$ which satisfy a growth condition; for all
$\epsilon \geq 0$ there exists $C_\epsilon$ such that
$$|s(q)|_K \leq C_\epsilon |q|^{\alpha-\epsilon}.$$

\noindent In general this extension is a coherent sheaf on which we do
not have much information, but Simpson shows in \cite{Simpsona} that
under some growth condition on the curvature of the metric, the metric
induces the above Deligne extensions. When a curvature satisfies this
condition it is called \textit{acceptable}.

\begin{lemma}[Lemma 2.4 \cite{EKMZ}]
  The local system $\mathbb V$ with non-expanding cusp monodromies has
  a metric which is acceptable for its Deligne extension $\mathcal V$
\end{lemma}

\begin{proof}
  For completeness, we reproduce the construction of \cite{EKMZ}.  The
  idea is to construct locally a nice metric and to patch the local
  constructions together with partition of unity.  The only delicate
  choice is for the metric around singularities.  We want the basis
  elements $\tilde v_i$ of the $\alpha$-eigenspace of the Deligne
  extension to be given the norm of order $|q|^\alpha$ in the local
  coordinate $q$ around the cusp and to be pairwise orthogonal.  Let
  $M$ be such that $e^{2i\pi M} = T$, where $T$ is the monodromy
  transformation.  Then the hermitian matrix
  $exp(log|q| \overline{M^t} M)$ defines a metric such that the
  element $\tilde v_i$ has norm $|q|^\alpha |\tilde v_i|$.
\end{proof}

\begin{corollary}
  \label{degenerate}
  When the monodromy representation goes to zero, the parabolic degree
  goes to zero.
\end{corollary}

\begin{proof}
  In the proof above, it is clear that when $T \to \mathrm{Id}$,
  $M \to 0$ and thus the metric goes to the standard hermitian metric
  locally. Thus its curvature goes to zero around singularities and
  its integral on any subbundle, which by definition is its parabolic
  degree, goes to zero.
\end{proof}

\subsection{Local Hodge invariants}
Our purpose in this subsection is to show the following relation on
local Hodge invariants:
\begin{theorem}
  \label{hodge invariants}
  The local Hodge invariants for equation \ref{hypergeometric} are :
  \begin{enumerate}
  \item at $z = 0$,
    $$ \nu_{\alpha_m}^p =
    \left\{
      \begin{array}{ll}
        1 \text{ if } p = f(\alpha_m)\\
        0 \text{ otherwise }\\
      \end{array}
    \right.
    $$
  \item at $z = \infty$,
    $$ \nu_{-\beta_m}^p =
    \left\{
      \begin{array}{ll}
        1 \text{ if } p - 1 = f(\beta_m)\\
        0 \text{ otherwise }\\
      \end{array}
    \right.
    $$
  \item at $z = 1$,
    $$ \nu_{\gamma}^p =
    \left\{
      \begin{array}{ll}
        1 \text{ if } p = [ \gamma ] + 1\\
        0 \text{ otherwise }\\
      \end{array}
    \right.
    $$
  \end{enumerate}
\end{theorem}

\begin{remark}
  Computations of (1) and (2) are done in \cite{Fedorov}. We give a
  similar proof with an alternative combinatoric point of view.\\
\end{remark}

\subsection{Computation of local Hodge invariants}
In the following, we denote by $M$ the local system defined by the
hypergeometric equation \ref{hypergeometric} in the introduction.  The
point at infinity plays a particular role in middle convolution, thus
we apply a biholomorphism to the sphere which will send the three
singularity points $0, 1, \infty$ to $ 0, 1, 2$. Hereafter, $M$ will
have singularities at $0, 1, 2$.

Similarly $M_{k,j}$ corresponds to the hypergeometric equation where
we remove terms in $\alpha_k$ and $\beta_j$,
$$\prod_{m \neq k} (D - \alpha_m) - z \prod_{n \neq j} (D - \beta_n) = 0$$

Let $L_{k,j}$ be a flat line bundle above
$\cpone - \{ 0, 2, \infty \}$ with monodromy $\Es{\alpha_k}$ at $0$,
$\Es{-\beta_j}$ at $2$ and $\Es{\beta_j - \alpha_k}$ at $\infty$.
Similarly $L'_{k,j}$ is defined to have monodromy $\Es{-\beta_j}$ at
$0$, $\Es{\alpha_k}$ at $2$ and $\Es{\beta_j - \alpha_k}$ at $\infty$.\\


The two key stones in the proof are Lemma $3.1$ in \cite{Fedorov} and
Theorem $3.1.2$ in \cite{Sabbah} :

\begin{lemma}[Fedorov]
  \label{middle convolution}
  For any $k,j \in \{ 1, \dots, n \}$ we have,
  $$ M \simeq MC_{\beta_j-\alpha_k} \left( M_{k,j} \otimes L'_{k,j} \right) \otimes L_{k,j}  $$
\end{lemma}

We modify a little bit the formulation of \cite{Sabbah}, taking
$\alpha = 1 - \alpha$. Thus the condition becomes
$1-\alpha \in (0, 1-\alpha_0] \iff \alpha \in [\alpha_0, 1)$.  Which
implies the following formulation.

\begin{theorem}[Dettweiler-Sabbah]
  \label{recursive}
  Let $\alpha_0 \in (0,1)$, for every singular point in $ \Delta $ and
  any $\alpha \in [0,1)$ we have,

  $$
  \def\arraystretch{1.5} \nu^p_{\alpha} \left( MC_{\alpha_0}(M)
  \right) = \left\{
    \begin{array}{l l}
      \nu^{p-1}_{\alpha-\alpha_0} \left( M \right) & \text{if } \ \alpha \in \left[0, \alpha_0 \right)\\
      \nu^{p}_{\alpha-\alpha_0} \left( M \right) & \text{if } \ \alpha \in \left[\alpha_0, 1 \right)
    \end{array}
  \right.
  $$

  \noindent and,
  $$
  \delta^p \left( MC_{\alpha_0}(M) \right) = \delta^p(M) + h^p(M)-
  \sum_{\underset{\{-\alpha\} \in [0, \alpha_0)}{s \in \Delta}}
  \nu_{s,\alpha}^{p-1} (M)
  $$
\end{theorem}

\subsubsection{Recursive argument}

We apply a recursive argument on the dimension of the hypergeometric
equation. Let us assume that $n \geq 3$ and that Theorem \ref{hodge
  invariants} is true for $n-1$. \\

For convenience in the demonstration, we change the indices of
$\alpha$ and $\beta$ such that $\alpha_i$ (resp. $\beta_i$) is the
$i$-th $\alpha$ (resp. $\beta$) we come upon while browsing
the circle to construct the function $f$.\\

We apply Lemma \ref{middle convolution} with $\alpha_k$ and $\beta_j$
such that $\alpha_k \leq \beta_j$.  Let us describe what happens to
the combinatorial function $f$ after we remove these two eigenvalues.
We denote by $f'$ the function we obtain.

Removing $\alpha_k$ will make the function decrease by one for the
following eigenvalues until we meet $\beta_j$, thus for any
$\tau \neq \alpha_k, \beta_j$,

$$
f(\tau) = \left\{
  \begin{array}{l l}
    f'(\tau)     & \text{if } \tau \prec \alpha_k \prec \beta_j \\
    f'(\tau) + 1 & \text{if } \alpha_k \prec \tau \prec \beta_j
  \end{array}
\right.
$$

\begin{figure}[h]
  \begin{subfigure}[c]{.45\linewidth}
    \centering
    \includegraphics[width=.9\linewidth]{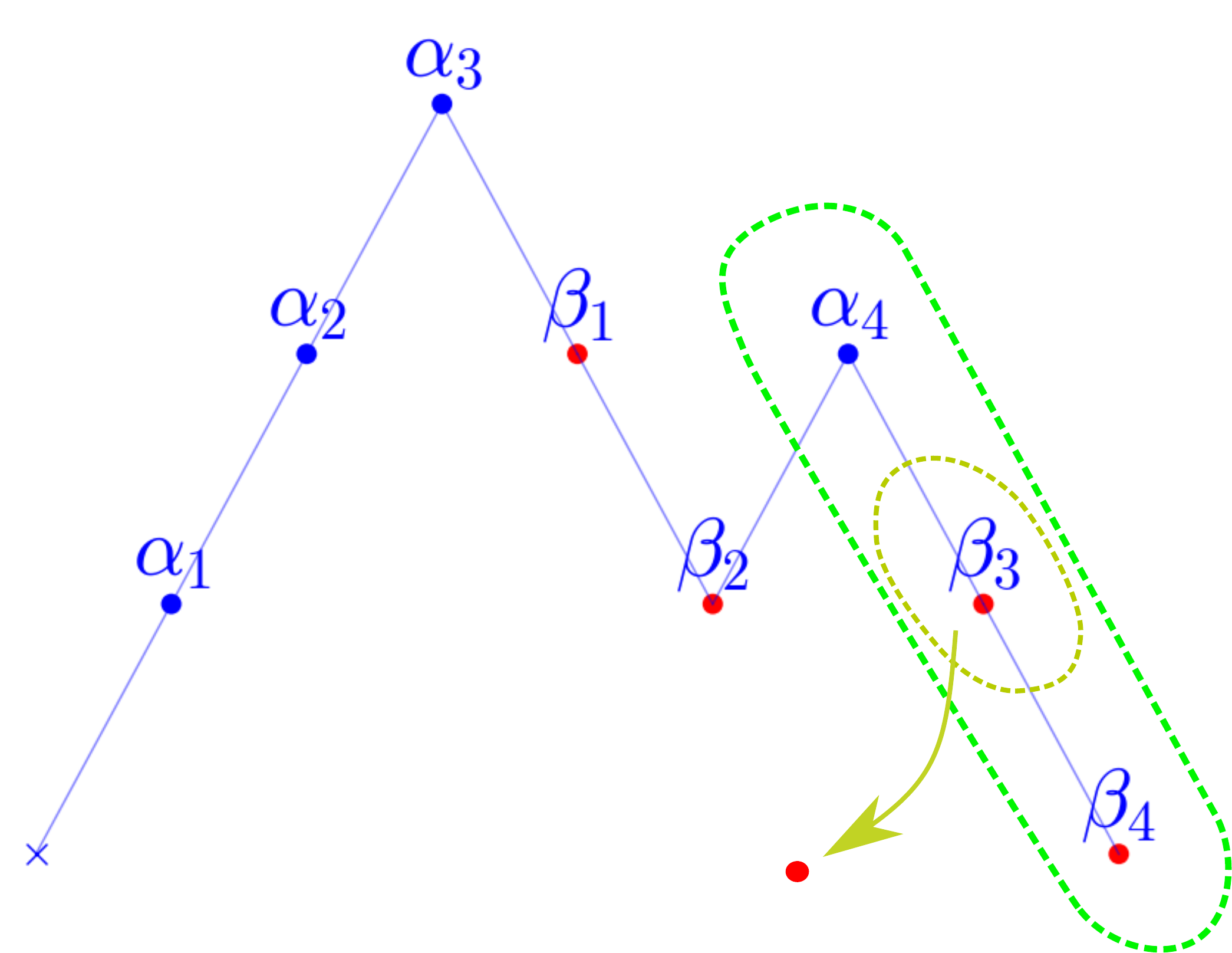}
    \subcaption{Graph of $f$}
  \end{subfigure}
  \begin{subfigure}[c]{.45\linewidth}
    \centering
    \includegraphics[width=.9\linewidth]{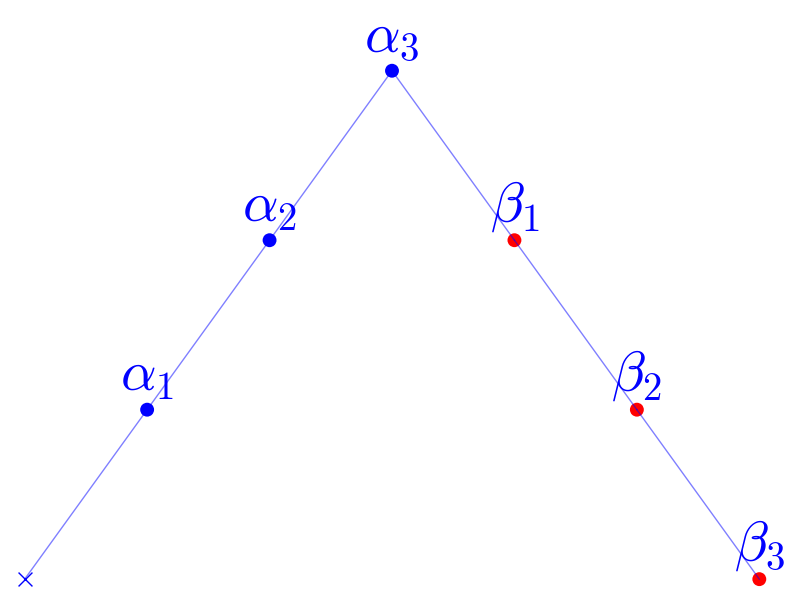}
    \subcaption{Graph of $f'$}
  \end{subfigure}
  \caption{Geometric representation of middle convolution action on
    $f$}
  \label{recprof}
\end{figure}

We apply Theorem \ref{recursive} with $\alpha_0 = \beta_j - \alpha_k$.
It yields that we have for all $m \neq k$, at singularity zero,
$$
\def\arraystretch{2} \nu^p_{1+\alpha_m-\alpha_k} \left( M \otimes
  L_{k,j}^{-1} \right) = \left\{
  \begin{array}{l l}
    \nu^{p-1}_{\alpha_m-\beta_j} \left( M_{k,j} \otimes L'_{k,j} \right) & \text{if }  \{ \alpha_m - \alpha_k \} <  \beta_j - \alpha_k\\
    \nu^{p}_{\alpha_m-\beta_j} \left( M_{k,j} \otimes L'_{k,j} \right) & \text{otherwise} 
  \end{array}
\right.
$$

\noindent which can be written in a simpler form
$$
\def\arraystretch{2} \nu^p_{\alpha_m-\alpha_k} \left( M \otimes
  L_{k,j}^{-1} \right) = \left\{
  \begin{array}{l l}
    \nu^{p-1}_{\alpha_m-\beta_j} \left( M_{k,j} \otimes L'_{k,j} \right) & \text{if } \alpha_k \prec \alpha_m \prec \beta_j\\
    \nu^{p}_{\alpha_k-\beta_j} \left( M_{k,j} \otimes L'_{k,j} \right) & \text{if }  \alpha_m \prec \alpha_k \prec \beta_j
  \end{array}
\right.
$$

\noindent In terms of $M$ and $M_{k,j}$,
$$
\def\arraystretch{2} \nu^p_{\alpha_m} \left( M \right) = \left\{
  \begin{array}{l l}
    \nu^{p-1}_{\alpha_m} \left( M_{k,j} \right) & \text{if } \alpha_k \prec \alpha_m \prec \beta_j\\
    \nu^{p}_{\alpha_m} \left( M_{k,j} \right) & \text{if }  \alpha_m \prec \alpha_k \prec \beta_j
  \end{array}
\right.
$$

\noindent For any integer $i,j$ we denote by $\delta( i, j)$ the
function which is $1$ when $i=j$ and is zero otherwise.

$$
\def\arraystretch{1.5} \nu^p_{\alpha_m} \left( M \right) = \left\{
  \begin{array}{l l}
    \delta \left(p-1, f'(\alpha_m) \right) & \text{if } \alpha_k \prec \alpha_m \prec \beta_j\\
    \delta \left(p, f'(\alpha_m) \right) & \text{if }  \alpha_m \prec \alpha_k \prec \beta_j
  \end{array}
\right.  \ = \delta \left( p, f(\alpha_m) \right)
$$

\noindent Similarly for all $m \neq j$, at singularity 2,
$$
\def\arraystretch{2} \nu^p_{-\beta_m} \left( M \right) = \left\{
  \begin{array}{l l}
    \nu^{p-1}_{-\beta_m} \left( M_{k,j} \right) & \text{if } \alpha_k \prec \beta_j \prec \beta_m  \\
    \nu^{p}_{-\beta_m} \left( M_{k,j} \right) & \text{if } \alpha_k \prec \beta_m \prec \beta_j 
  \end{array}
\right.
$$

$$
\def\arraystretch{1.5} \nu^p_{\beta_m} \left( M \right) = \left\{
  \begin{array}{l l}
    \delta \left( p-1, f'\left( \beta_m \right) \right) & \text{if } \alpha_k \prec \beta_j \prec \beta_m\\
    \delta \left( p, f'\left( \beta_m \right) \right) & \text{if } \alpha_k \prec \beta_m \prec \beta_j
  \end{array}
\right.  = \delta \left( p,f \left( \beta_k \right) \right)
$$

\noindent And at 1, we set
$\tilde{\gamma} := \gamma-\beta_j+\alpha_k = \sum_{m \neq j} \beta_m -
\sum_{m \neq k} \alpha_m$,
$$
\def\arraystretch{2} \nu^p_{\gamma} \left( M \right) = \left\{
  \begin{array}{l l}
    \nu^{p-1}_{\tilde{\gamma}} \left( M_{k,j} \right) & \text{if } \{\gamma\} < \beta_j - \alpha_k\\
    \nu^{p}_{\tilde{\gamma}} \left( M_{k,j} \right) & \text{otherwise}
  \end{array}
\right.
$$

$$
\def\arraystretch{2} \nu^p_{\gamma} \left( M \right) = \left\{
  \begin{array}{l l}
    \delta \left( p-1, [ \tilde \gamma] + 1 \right) & \text{if } \{\gamma\} = \{\tilde{\gamma}\} + \beta_j - \alpha_k - 1\\
    \delta \left( p, [\tilde \gamma] + 1 \right)  & \text{if }  \{\gamma\} = \{\tilde{\gamma}\} + \beta_j - \alpha_k
  \end{array}
\right.  \ = \delta \left(p, [\gamma] + 1 \right)
$$

Now if we choose to pick $\beta_n > \alpha_n$ for the computation, we
now hodge invariants for all values excepts for $\alpha_n$ and
$\beta_n$.  We will use the computation with for example
$\alpha_1 < \beta_1$.  From this one we can deduce the invariants at
$\beta_n$ and $\alpha_n$.  Yet, we should keep in mind that the
previous computations are always modulo shifting of the VHS. That is
why we need to have dimension at least $3$, since in this case
$\alpha_2$ will appear in both computations and will show there is no
shift in our formulas.

\subsubsection{Initialization for n = 2}
We use the computations performed in the previous part for $\alpha_2$
and $\beta_2$. To do so, first remark that the unique (complex
polarized) VHS on $M_{2,2}$ is defined by
$h^p \left( M \right) = \delta(p,1)$ and the only non-zero local Hodge
invariants are
\[
  \begin{array}{l l l}
    \nu_{\alpha_1}^1 \left( M_{2,2} \right) &= 1 & \text{at singularity } 0\\
    \nu_{-\beta_1}^1 \left( M_{2,2} \right) &= 1 & \text{at singularity } \infty\\
    \nu_{\beta_1-\alpha_1}^1 \left( M_{2,2} \right) &= 1 & \text{at singularity } 1\\
  \end{array}
\]
\noindent Which corresponds to the definition of
$\delta(p,f'(\alpha_1))$ for the first two,
and to $\delta(p, [\gamma]+1)$ for the last one.\\

\noindent Using the previous subsection, we deduce
\[
  \begin{array}{l l l}
    \nu_{\alpha_1}^p \left( M \right) &= \delta(p,f(\alpha_1)) & \text{at singularity } 0\\
    \nu_{-\beta_1}^p \left( M \right) &= \delta(p,f(\alpha_1)) & \text{at singularity } \infty\\
    \nu_{\gamma}^p \left( M \right) &= \delta(p,[\gamma] + 1
                                      ) & \text{at singularity } 1\\
  \end{array}
\]

\noindent According to \cite{Fedorov}, the Hodge numbers on $M$ are
$$
h^1 = \left\{
  \begin{array}{l l}
    1 & \text{if } \alpha_1 \prec \alpha_2 \prec \beta_1 \prec \beta_2 \\
    2 & \text{if } \alpha_1 \prec \beta_1 \prec \alpha_2 \prec \beta_2 
  \end{array}
\right., \hspace{1cm} h^2 = \left\{
  \begin{array}{l l}
    1 & \text{if } \alpha_1 \prec \alpha_2 \prec \beta_1 \prec \beta_2\\
    0 & \text{if } \alpha_1 \prec \beta_1 \prec \alpha_2 \prec \beta_2
  \end{array}
\right.
$$

\noindent Using the fact that $\sum_{\alpha} \nu^p_{\alpha} = h^p$, we
can deduce the other Hodge invariants.
$$
\begin{array}{l l}
  \nu_{\alpha_2}^2 = 1 & \text{if } \alpha_1 \prec \alpha_2 \prec \beta_1 \prec \beta_2\\
  \nu_{\alpha_2}^1 = 1 & \text{if } \alpha_1 \prec \beta_1 \prec \alpha_2 \prec \beta_2
\end{array}
$$

\noindent We conclude that
$\nu_{\alpha_2}^p \left( M \right) = \delta(p,f(\alpha_2))$ and
similarly $\nu_{-\beta_2}^p \left( M \right) = \delta(p,f(\alpha_2))$.

\subsection{Continuity of the parabolic degree}

To compute $\delta^p(V)$ in equation \ref{parabolic degree}, we show
in the following Lemma a continuity property which implies that it is
constant on a given domain.

\begin{lemma}
  Let $\alpha_1, \dots, \alpha_n, \beta_1, \dots, \beta_n$ be all
  disjoint, not integers and such that $\gamma$ also is not an
  integer.  Fix the intertwining diagram of $\alpha$ and $\beta$ and
  the value of $[\gamma]$, then for any integer $p$, $\delta^p(V)$ is
  constant.
\end{lemma}

\begin{proof}
  Let $L$ and $L'$ be the local system corresponding to equation
  \ref{hypergeometric} for eigenvalues $\alpha,\beta$ (resp.
  $\alpha', \beta'$) satisfying the above hypothesis. We endow them
  with a trivial filtration. According to Simpson's theory, $L$
  corresponds to some Higgs bundle $(E, \theta)$ together with a
  parabolic structure at singularities. As $L$ has eigenvalues of norm
  one, $\theta$ has no residue, moreover its weight filtration is
  locally the same as the one for the unipotent part of monodromy
  matrices of $L$.

  Consider now $(E', \theta')$ a Higgs bundle with the same
  holomorphic structure and Higgs form as $(E, \theta)$ but a slightly
  changed parabolic structure for which we keep the initial filtration
  but modify the parabolic weights $\alpha, \beta$ to
  $\alpha', \beta'$. It is clear that we keep the same residues
  $\mathrm{res}_s(\theta) = \mathrm{res}_s(\theta')$ at every
  singularity $s$. We also keep locally the same weight filtration for
  the unipotent part of the monodromy on any eigenspace. This implies
  that monodromy matrices of the local system associated to
  $(E', \theta')$ and those of $L'$ are locally isomorphic, and by
  rigidity of the hypergeometric local systems they are globally
  isomorphic. The same argument applies for Hodge subbundles.  As the
  considered domain is connected this shows that the parabolic degree
  of the Hodge subbundles is constant.
\end{proof}

Together with Corollary \ref{degenerate} this is enough to compute
$\delta^p(V)$.  We fix an intertwining diagram and a floor value for
$\gamma$ and make the first $[\gamma]$ $\beta$ go to $1$ and the rest
of eigenvalues to $0$ while staying in the given domain.  At the
limit, the parabolic degree is zero and we can deduce $\delta^p(V)$.

\section{Algorithm}
In this section, we describe the algorithm used to compute the
Lyapunov exponents.  We start simulating a generic hyperbolic geodesic
and following how it winds around the surface, namely the evolution of
the homology class of the closed path.  Finally we compute the
corresponding monodromy matrix after each turn around a cusp.

\subsection{Hyperbolic geodesics}
This first question arising to unravel this computation of Lyapunov
exponents is how to simulate a generic hyperbolic geodesic.  The
answer comes from a beautiful theorem proved by Caroline Series in
\cite{Series1985} which relates hyperbolic geodesics on the Poincaré
half-plane and continued fraction development of real numbers.
We follow here the notations of \cite{Dalbo} (see part II.4.1).\\

Let us consider the Farey tessellation of $\mathbb H$ (see Figure
\ref{Farey}).
\begin{figure}
  \centering
  \includegraphics[width=\linewidth]{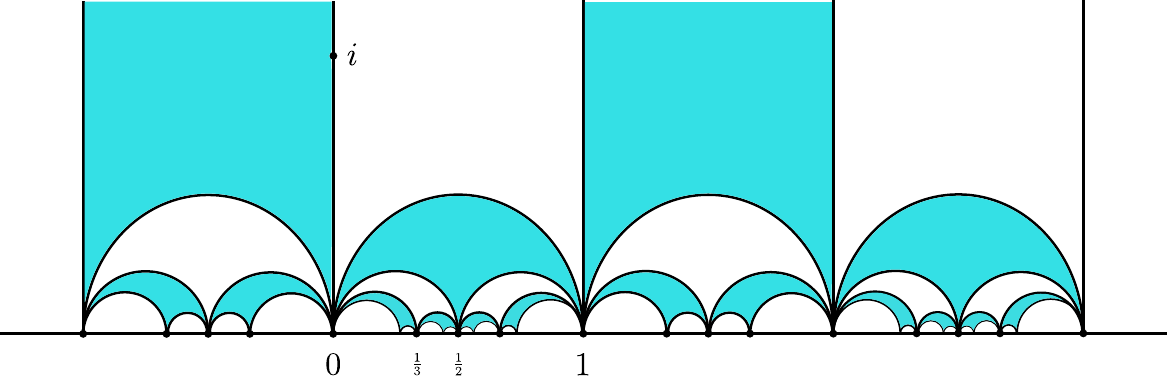}
  \caption{The Farey's tessellation}
  \label{Farey}
\end{figure}
It is the fundamental domain for the discrete subgroup of index $3$ in
$\PSLZ$ generated by
$$
\left< \left(
    \begin{array}{cc}
      1 & 1 \\
      0 & 1 \end{array}
  \right),
  \left(
    \begin{array}{cc}
      1 & 0 \\
      1 & 1
    \end{array}
  \right) \right>
$$

The sphere minus three points endowed with its complete hyperbolic
metric is a degree two cover of the surface associated to Farey's
tessellation.  This is why we represent the tessellation with two
colors : the fundamental domain for the sphere corresponds to two
adjacent triangles of different colors. That is why it will be easy
once we understand the geodesics with respect to this tessellation
to see them on the sphere.\\

Let us consider a geodesic going through $i$. It lands to the real
axis at a positive and a negative real number. The positive real
number will be called $x$, this number determines completely the
geodesic since we know two distinct points on it.

We associate to this geodesic a sequence of positive integers.  Look
at the sequence of hyperbolic triangles the geodesic will cross.  For
each one of those triangles, the geodesic has two ways to cross them
(see Figure 3). Once it enters it, it can leave it crossing either the
side of the triangle to its left (a) or to its right (b).

\begin{figure}[h]
  \begin{subfigure}[c]{.5\linewidth}
    \centering
    \includegraphics[height=2cm]{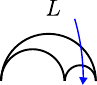}
    \subcaption{to the left}
  \end{subfigure}
  \begin{subfigure}[c]{.5\linewidth}
    \centering
    \includegraphics[height=2cm]{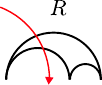}
    \subcaption{to the right}
  \end{subfigure}
  \caption{Two ways to cross an hyperbolic triangle}
\end{figure}

\begin{remark}
  The vertices of hyperbolic triangles are located at rational
  numbers, so this sequence will be infinite if and only if $x$ is
  irrational (see \cite{Dalbo} Lemme 4.2).
\end{remark}

We have now for a generic geodesic an infinite word in two letters $L$
and $R$ associated to a geodesic. For example the word associated to
the geodesic in Figure \ref{geodesic example}, is of the form
$LLRRLR\dots = L^2R^2L^1R\dots$. We can factorize each of those words
and get $$R^{n_0}L^{n_1}R^{n_2}L^{n_3}\dots$$ Except for $n_0$ which
can be zero the $n_i$ are positive integers.

\begin{theorem}
  The sequence $(n_k)$ is the continued fraction development of $x$.
  In other words,
  $$ x = n_0 + \cfrac 1 {n_1 + \cfrac 1 {n_2 + \cfrac 1 {n_3 + \dots}}} $$
  The measure induced on the real axis by the measure on
  $T^1 \mathbb H$ dominates Lebesgue measure.
\end{theorem}

See \cite{Dalbo} II.4 or \cite{Series1985} for a proof.

\begin{remark}
  This theorem states exactly that to study a generic geodesic on the
  hyperbolic plane, we can consider a Lebesgue generic number in
  $(0, \infty)$ and compute its continued fraction development.
\end{remark}

\begin{figure}[h]
  \begin{subfigure}[t]{\linewidth}
    \centering
    \includegraphics[width=\linewidth]{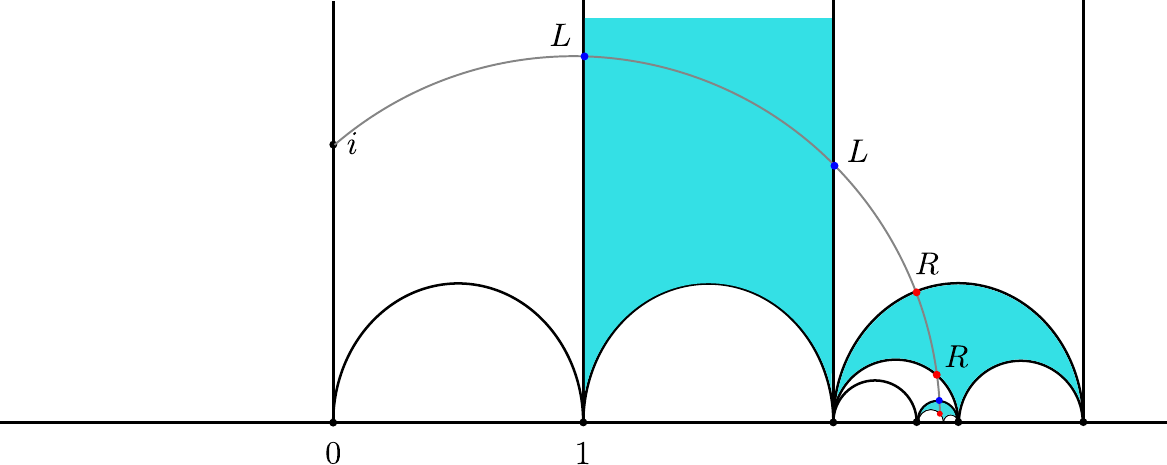}
  \end{subfigure}
  \caption{Crossings of a given geodesic}
  \label{geodesic example}
\end{figure}

To compute Lyapunov exponents of the flat bundle, we need to follow
how a generic geodesic winds around the cusps.  By the previous
theorem we can simulate a generic cutting sequence of an hyperbolic
geodesic in $\Hy$.  Our goal now will be to associate to such a
sequence a product of monodromy
matrices following its homotopy class.\\

Since we will consider universal cover of the sphere minus three
points, for convenience we will denote by $A$, $B$, $C$ the cusps
corresponding in the surface to $\infty$, $0$, $1$ respectively and
use this latter notation for points in $\Hy$. Two adjacent hyperbolic
triangles of Farey's tessellation, e.g. $0, 1, \infty$ and
$1, 2, \infty$ will form a fundamental domain for this surface. All
the vertices of the hyperbolic triangles for this tessellation are
associated to either $A, B$ or $C$. To follow how the flow turns
around these vertices in the surface, we will need to keep track of
orientation. To do so, we color the triangles according to the order
of its vertices, when we browse the three vertices counterclockwise if
we have $A \rightarrow B \rightarrow C \rightarrow A$ we color the
triangle in white (this is the case for $0,1,\infty$) otherwise
we color it in blue (case of $\infty, 2, 1$).\\

Let us now consider a point $P$ inside the blue triangle, which will
be used as a base point for an expression of the cycles around cusps.
We choose a homology marking of the surface by denoting the paths
going around $A$, $B$, $C$ counterclockwise starting and ending at
$P$, \; $a$, $b$, $c$ (see Figure \ref{marking}).  When we concatenate
these paths we get $c \cdot b \cdot a = \Id$ and $a^{-1} = c \cdot b$.
For monodromy matrices we will have the relation
\begin{equation}
  \label{monodromy}
  M_\infty M_0 M_1 = \Id
\end{equation}\\

\begin{figure}[h]
  \begin{subfigure}[c]{0.45\linewidth}
    \centering
    \includegraphics[width=.8\linewidth]{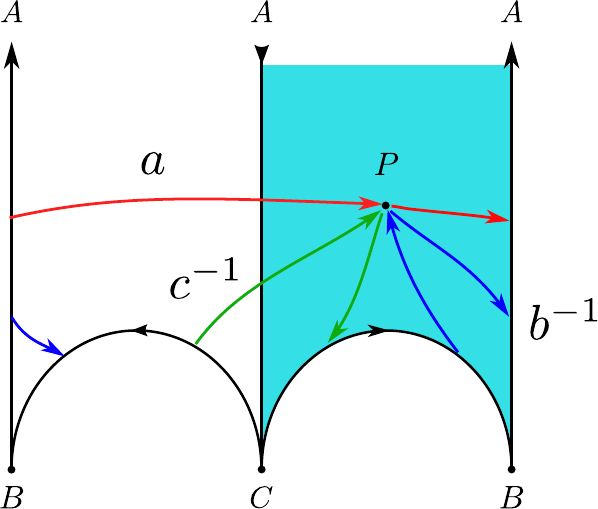}
  \end{subfigure}
  \begin{subfigure}[c]{.45\linewidth}
    \centering
    \includegraphics[width=.8\linewidth]{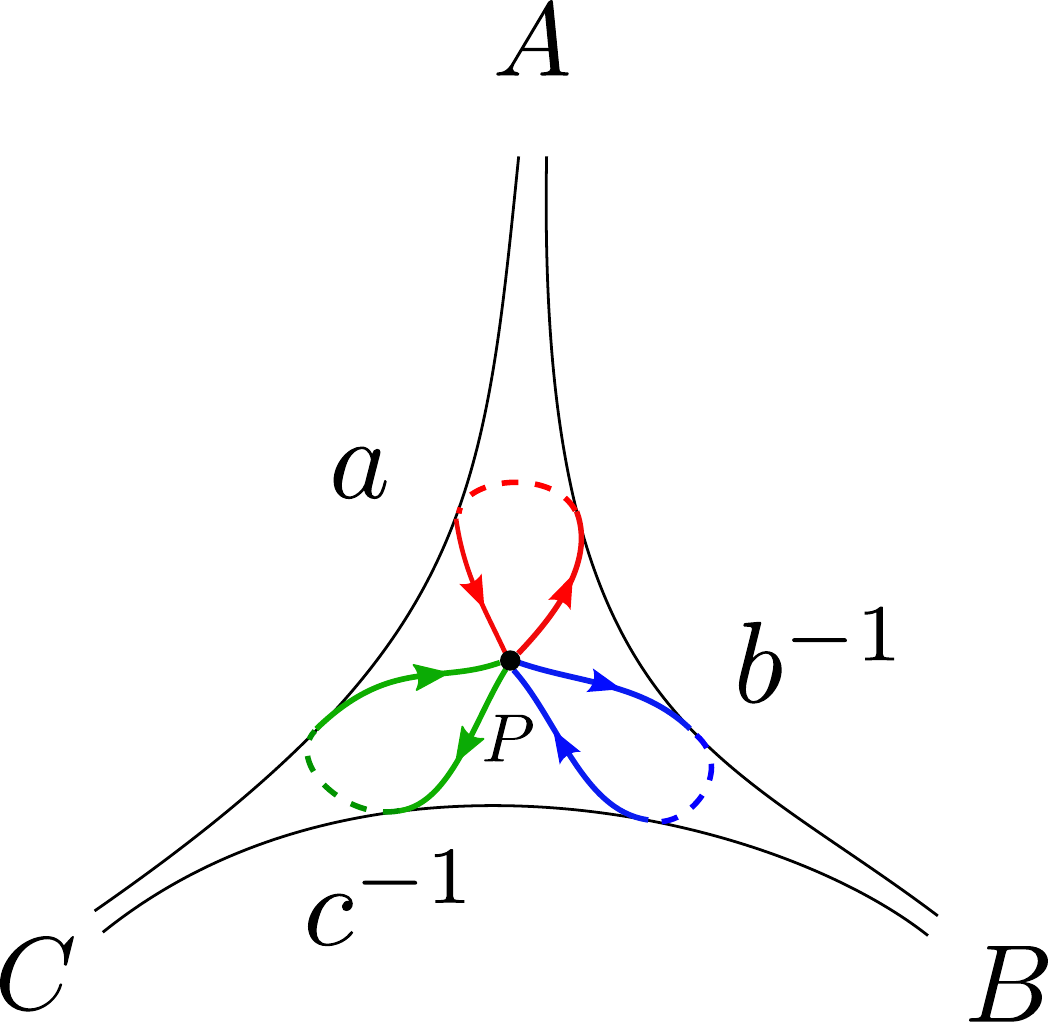}
  \end{subfigure}
  \caption{Homology marking}
  \label{marking}
\end{figure}

In our algorithm we will always follow the cutting sequence until we
end up to a blue triangle. Then we will apply an isometry that take
the fundamental domain we are in to the $(-1, 0, \infty)$ triangle and
the edge the flow will cut when going out of the triangle to be the
$(0 \infty)$ or $(-1, \infty)$ edge in order to place the cusp we are
turning around at $\infty$. We shall warn the reader here that the
corresponding cusp on the surface here at points $-1, 0$ and $\infty$
may be any of the points $A, B, C$ but their cyclic order will stay
unchanged thanks to the orientation. Thus we just need to keep track
of the cusp placed at $\infty$.

When we start with a cutting sequence extracted from the previous
theorem we see that the geodesic start by cutting $(0,\infty)$ at $i$
without being counted in the cutting sequence. The first cutting will
always be forgotten in the sequence when applying the isometry.

Now remark that when the crossing is a sequence of $2n$ left, we make
$n$ turns counterclockwise around the cusp placed at $\infty$. When it
$2n$ right, we make $n$ turn clockwise.  It is a little trickier if
the geodesic makes an odd number of the same crossing; we need to take
one step further from the next term in the sequence of
crossings to end up at $P$ (see Figure \ref{odd}).\\

\begin{figure}[h]
  \begin{subfigure}[c]{.45\linewidth}
    \centering
    \includegraphics[width=.9\linewidth]{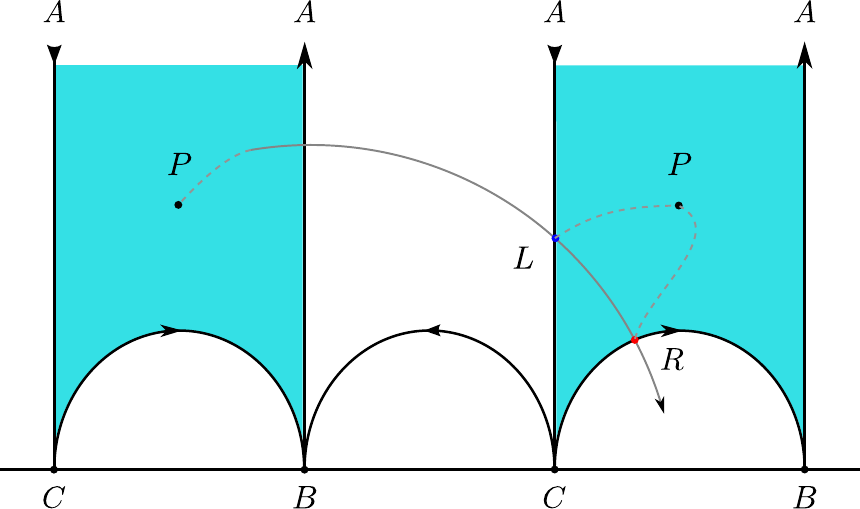}
    \subcaption{Odd number of sections}
  \end{subfigure}
  \begin{subfigure}[c]{.45\linewidth}
    \centering
    \includegraphics[width=.9\linewidth]{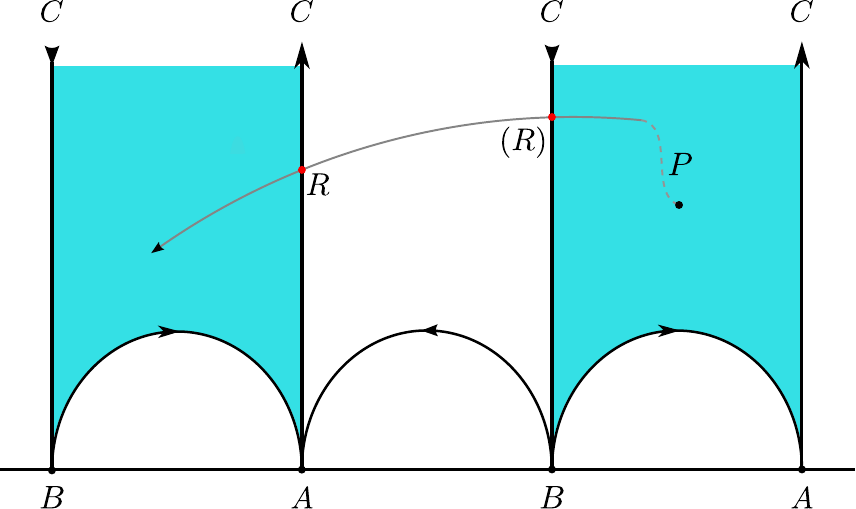}
    \subcaption{Changing the setting}
  \end{subfigure}
  \caption{Applying the good orientation preserving isometry}
  \label{odd}
\end{figure}

There is a last point to consider, since we want to compare the growth
of the harmonic norm with regards to the geodesic flow, need to follow
its length.  Here the discretized algorithm enables us to follow the
type of homotopy it will have, but the length will not correspond a
priori to the number of iterations of our algorithm.  It is
proportional to it by the constant.

\subsection{Monodromy matrices}
\label{secmonodromy}

In the introduction Proposition \ref{eigenvalues} gave a set of three
properties on the monodromy matrices for the hypergeometric
differential equation associated to two distinct sequences of real
numbers $\alpha_1, \dots, \alpha_n$ and $\beta_1, \dots, \beta_n$.  We
claim that those properties are sufficient to recover the monodromy
matrices up to conjugacy.

For convenience we always assume that the $\alpha_1, \dots, \alpha_n$
are disjoint, otherwise the computation becomes way more tedious, and
in our computations we will explore generic domains.  We choose a
basis in which $M_0$ is diagonal.  Property (3) tells us that
$M_1 - \Id$ is of rank $1$. We can then find two vectors $v$ and $w$
such that $M_1 = \Id + v w^t$.

Since $M_\infty^{-1} = M_0 M_1$ knowing the eigenvalues of $M_\infty$
we can derive the following $n$ equations, for all $j$,

\[
  \det \left(M_\infty^{-1} - \E{\beta_j} \Id \right) = 0\\
\]

We can compute this determinant using the particular form of the
matrix and the following lemma.
\[
  M_0 M_1 - \E{\beta_j} \Id = (M_0 - \E{\beta_j} \Id) + (M_0 v) w^t
\]

We can conjugate by diagonal matrices so that $M_0 v$ becomes the
vector $\bold 1$ which is one on every coordinates. And obtain the
equations

\[
  \det ((M_0 - \E{\beta_j} \Id) + \bold 1 w^t) = 0, \; \forall j
\]

\begin{lemma*}
  Let $D$ a diagonal matrix with $d_1, \dots, d_n$ on its diagonal,
  and $x$ a vector.
  $$\det \left( D + \bold 1 x^t \right) = \left(\prod_{i=1}^{n} d_i\right) \cdot \left(1 + \sum_{i=1}^n x_i/d_i \right)$$
\end{lemma*}

\begin{proof}
  First consider the case where $D$ is the identity matrix. We know
  that all the eigenvalues except for one are 1. The determinant will
  then be the eigenvalue of an eigenvector which image through $x^t$
  is not zero.  This vector will be $\bold 1$ and its eigenvalue
  $(1 + \sum_{i=1}^n x_i)$.  To finish the proof, just factor each
  column by $d_i$ in the determinant.
\end{proof}

We obtain
  $$\left(\prod_{i=1}^n \E{\alpha_i} - \E{\beta_j} \right) 
  \left( 1 + \sum_{i=1}^n \frac{w_i}{\E{\alpha_i} - \E{\beta_j}}
  \right) = 0$$

\begin{corollary*}
  The vector $w$ satisfies for all $j$,
  $$ \sum_{i=1}^n \frac{w_i}{\E{\beta_j} - \E{\alpha_i}} = 1 $$
\end{corollary*}

We define a matrix
$N = \left(\frac{1}{\E{\beta_j} - \E{\alpha_i}}\right)_{i,j}$ and
observe that $w^t N = \bold 1^t$ so $w^t = \bold 1^t N^{-1}$.  Hence
for a generic setting, we just have to invert $N$ to find the explicit
monodromies.
And finally we have the expression $M_1 = \Id + M_0^{-1} \bold 1 \cdot \bold 1 ^t N^{-1}$\\

\section{Observations}

\subsection{Calabi-Yau families example}
A first family of examples is coming from 14 1-dimensional families of
Calabi-Yau varieties of dimension $3$.  The Gauss-Manin connection for
this family on its Hodge bundle gives an example of the hypergeometric
family we are considering. The monodromy matrices were computed
explicitly in \cite{Enckevort2008} and have a specific form
parametrized by two integers $C$ and $d$.  We introduce the following
monodromy matrices,

$$T = \left(
  \begin{array}{cccc}
    1   & 0   & 0 & 0 \\
    1   & 1   & 0 & 0 \\
    1/2 & 1   & 1 & 0 \\
    1/6 & 1/2 & 1 & 1 \\
  \end{array}
\right) \hspace{1cm} S = \left(
  \begin{array}{cccc}
    1   & -C/12 & 0 & -d \\
    0   &     1 & 0 &  0 \\
    0   &     0 & 1 &  0 \\
    0   &     0 & 0 &  1 \\
  \end{array}
\right)$$

\vspace{.3cm}

\noindent In the previous notations,
$M_0 = T, M_1 = S, M_\infty = (TS)^{-1}$.  These matrices satisfy
relation \ref{monodromy}, $M_\infty M_0 M_1 = \Id$.  We see that
$M_1-\Id$ has rank one and eigenvalues of $M_0$ and $M_\infty$ have
module one thus correspond to hypergeometric equations.  In this
setting, $T$ has eigenvalues all equal to one and eigenvalues of
$(TS)^{-1}$ are symmetric with respect to zero, we denote them by
$\mu_1, \mu_2, -\mu_2, -\mu_1$ where $\mu_1, \mu_2 \geq 0$. \\

The parabolic degree of the holomorphic Hodge subbundles are given by,
\begin{theorem*}{\cite{EKMZ}}
  Suppose $0 < \mu_1 \leq \mu_2 \leq 1/2$ then the degree of the Hodge
  bundles are
  $$ \deg_{par} \mathcal E ^{3,0} = \mu_1 \ \text{ and } \ \deg_{par} \mathcal E^{2,1} = \mu_2$$
\end{theorem*}

Thus according to the same article, we know that $2(\mu_1 + \mu_2)$ is
a lower bound for the sum of Lyapunov exponents.  We call good cases
the equality cases and bad cases the cases where
there is strict inequality.\\

There are 14 different couples of values for $C$ and $d$ where the
corresponding flat bundle is an actual Hodge bundle over a family of
Calabi-Yau varieties. These examples where computed few year ago by
M. Kontsevich and were a motivation for this article. We list them in
the table below.

\begin{figure}[htt!]
  \begin{subfigure}[c]{.45\linewidth}
    \centering \hspace*{-.5cm}
    \begin{tabular}{|c|c|c|c|c|}
      \hline
      C  & d  & $\lambda_1 + \lambda_2$ & $\lambda_1$ & $\mu_1, \mu_2$ \\ \hline \hline
      46 & 1  & 1                      & 0.97        & 1/12, 5/12     \\ \hline
      44 & 2  & 1                      & 0.95        & 1/8, 3/8       \\ \hline
      52 & 4  & 4/3                    & 1.27        & 1/6, 1/2       \\ \hline
      50 & 5  & 6/5                    & 1.12        & 1/5, 2/5       \\ \hline
      56 & 8  & 3/2                    & 1.40        & 1/4, 1/2       \\ \hline
      60 & 12 & 5/3                    & 1.53        & 1/3, 1/2       \\ \hline
      64 & 16 & 2                      & 1.75        & 1/2, 1/2       \\ \hline
    \end{tabular}
    \subcaption{The 7 good cases}
  \end{subfigure}
  \begin{subfigure}[c]{.45\linewidth}
    \centering \hspace*{.5cm}
    \begin{tabular}{|c|c|c|c|c|}
      \hline
      C  & d  & $\lambda_1 + \lambda_2$ & $\lambda_1$ & $\mu_1, \mu_2$ \\ \hline \hline
      22 & 1  & 0.92                   & 0.75        & 1/6, 1/6       \\ \hline
      34 & 1  & 0.83                   & 0.77        & 1/10, 3/10     \\ \hline
      32 & 2  & 0.97                   & 0.84        & 1/6, 1/4       \\ \hline
      42 & 3  & 1.06                   & 0.96        & 1/6, 1/3       \\ \hline
      40 & 4  & 1.30                   & 1.07        & 1/4, 1/4       \\ \hline
      48 & 6  & 1.31                   & 1.15        & 1/4, 1/3       \\ \hline
      54 & 9  & 1.60                   & 1.34        & 1/3, 1/3       \\ \hline
    \end{tabular}
    \subcaption{The 7 bad cases}
  \end{subfigure}
  \caption{Experiments}
\end{figure}

To see what happens in a similar setting for more general
hypergeometric equations, we vary $C, d$ and compute the corresponding
eigenvalues $\mu_1$ and $\mu_2$ as well as the Lyapunov exponents.  On
Figure \ref{bad and good} we drew a blue point at coordinate
$(\mu_1,\mu_2)$ if the sum of positive Lyapunov exponents are as close
to the parabolic degree $2(\mu_1 + \mu_2)$ as the precision we have
numerically and we put a red point when this value is outside of the
confidence interval.

\begin{figure}[h]
  \begin{subfigure}[c]{.45\linewidth}
    \centering \hspace*{-1cm}
    \includegraphics[height=5cm]{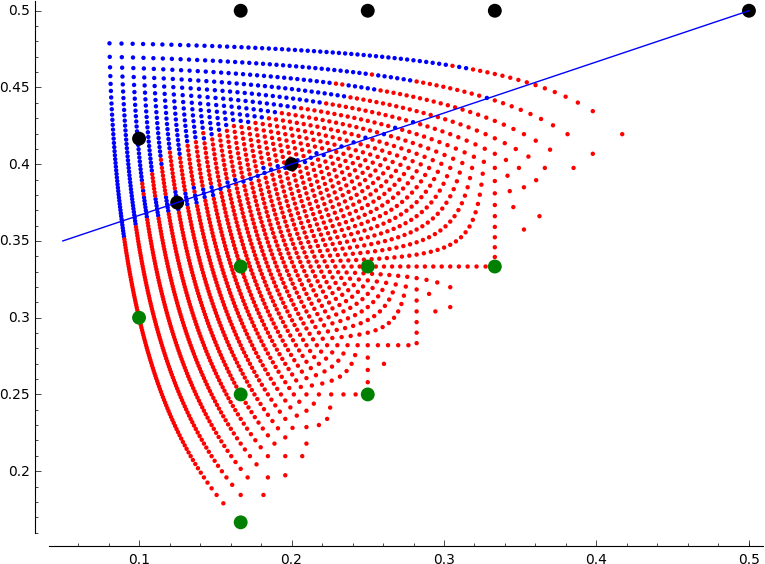}
    \subcaption{The good and bad cases}
    \label{bad and good}
  \end{subfigure}
  \begin{subfigure}[c]{.45\linewidth}
    \centering \hspace*{1cm}
    \includegraphics[height=5cm]{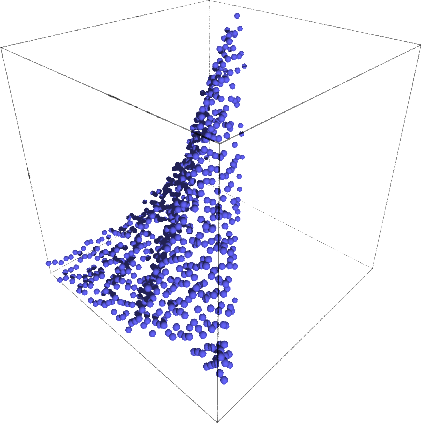}
    \subcaption{Zoom on the part above the line}
    \label{good}
  \end{subfigure}
  \caption{Experiments}
\end{figure}

Note that according to Figure \ref{bad and good} it seems that all
points below the line of equation $3\mu_2 = \mu_1 + 1$ are bad cases.
In Figure \ref{good}, we represent the distance of the sum of the
Lyapunov exponents to the expected formula. We see that this gives a
function that oscillates above zero. More precisely, it seems that
good cases are outside of some lines passing through $(1/2,1/2)$.

To push the numerical simulations further, we consider what happens on
lines of equation $3\mu_2 = \mu_1 + 1$ \ref{line_one} and
$48\mu_2 = 10\mu_1 + 19$ \ref{line_two} both passing through
$(1/2, 1/2)$ and a point corresponding to one of the previous good
cases.

\begin{figure}[h]
  \begin{subfigure}[c]{.45\linewidth}
    \includegraphics[height=4cm]{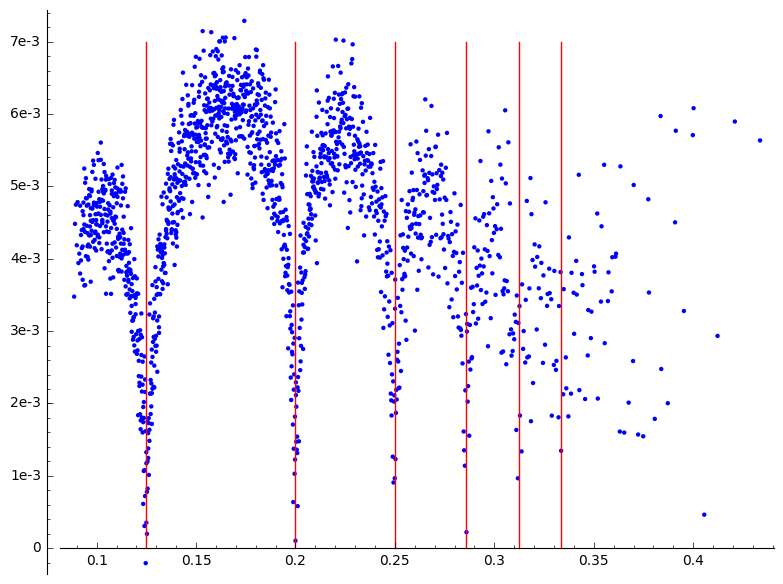}
    \subcaption{ $3\mu_2 = \mu_1 + 1$ }
    \label{line_one}
  \end{subfigure}
  \begin{subfigure}[c]{.45\linewidth}
    \includegraphics[height=4cm]{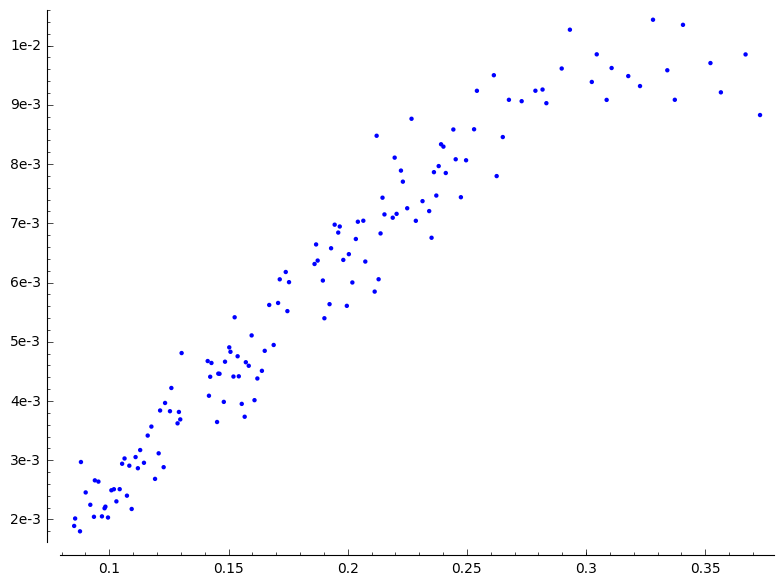}
    \subcaption{$48\mu_2 = 10\mu_1 + 19$}
    \label{line_two}
  \end{subfigure}
  \caption{Lyapunov exponents in function of $\mu_1$}
\end{figure}

We observe that on the graph \ref{line_two} there is only one good
case which corresponds to $(\mu_1, \mu_2) = (1/10, 3/10)$ in the
previous list of good cases. In the graph \ref{line_one}, there are
good cases at points $(\mu_1, \mu_2) = (1/8, 3/8), (1/5, 2/5)$ which
were also on the previous list but other points appear such as
$(3/12, 5/12), (5/16, 7/16), (3/9, 4/9)$.

According to \cite{BT14} and \cite{SV14}, the 7 good cases correspond
to cases where the monodromy group of the hypergeometric local system
is of infinite index in $Sp(4,\Z)$, which is commonly called
\textit{thin}. In the other cases the group is of finite index and is
called \textit{thick}.  The three good cases we found by ways of
Lyapunov exponents do not seem to have a representation with integers
$C$ and $d$. A lot of questions arise about these points, for example
can we find a number-theoretic interpretation of their equality as in
Conjecture 6.5 in \cite{EKMZ}.

\subsection{Examples for $n=2$}
\label{n=2}

Has we have seen in the introduction the two Lyapunov exponents are
symmetric $\lambda_1$ and $-\lambda_1$.  The sum of the positive
Lyapunov exponents is just $\lambda_1$. The parameter space we have
for these $2$-dimensional flat bundles are
$\alpha_1, \alpha_2, \beta_1, \beta_2$.

The Lyapunov exponents are invariant through translation of the set of
parameters. Indeed, we can consider the bundle with $e^\delta M_0$ and
$e^{-\delta} M_\infty$ monodromies, it will have the same set of
Lyapunov exponents since both scalar will appear with the same
frequency and its parameters will be
$\alpha_1 + \delta, \dots, \alpha_h + \delta, \beta_1 + \delta, \dots,
\beta_h + \delta$ hence without loss of generality we can assume
$\beta_1 = 0$.  Moreover the parameters are given as a set, the order
does not matter.

In the following experiments we will consider a set of parameters
where the $\beta$'s will be equidistributed and the $\alpha$'s will be
shifted with respect to them.  Here we represent the value of the
Lyapunov exponent for
$\alpha_1 = r, \alpha_2 = 2r, \beta_1 = 0, \beta_2 = x$ and we have by
definition $\gamma = x-3r$.

\begin{figure}[h]
  \begin{subfigure}[c]{.45\linewidth}
    \centering
    \includegraphics[width=\linewidth]{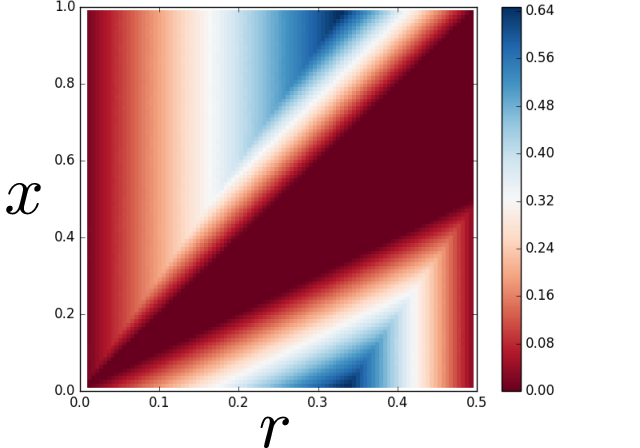}
    \subcaption{Plotting $\lambda_1$}
    \label{plot}
  \end{subfigure}
  \begin{subfigure}[c]{.45\linewidth}
    \centering
    \includegraphics[width=\linewidth]{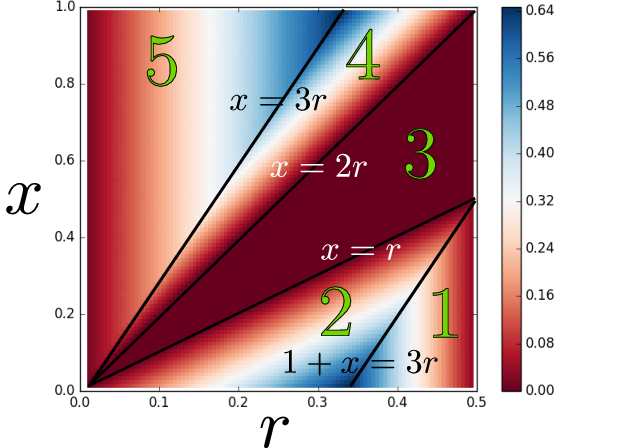}
    \subcaption{Zones on the graph}
    \label{zones}
  \end{subfigure}
  \caption{Experiments}
\end{figure}

\begin{remark}
  We first notice that the zone where the Lyapunov exponent is zero
  corresponds to the setting where the parameters are alternate and
  where there is a positive definite bilinear form invariant by the
  flat connection (see introduction).  This will be true whenever the
  VHS has weight $0$.

  Another noticeable fact is that zones correspond exactly to
  different combinatorics for the order of the $\alpha$ and $\beta$,
  and on $[ \gamma ]$ introduced in the introduction.

  \begin{figure}[h]
    \begin{subfigure}[c]{.45\linewidth}
      \centering
      \includegraphics[width=.9\linewidth]{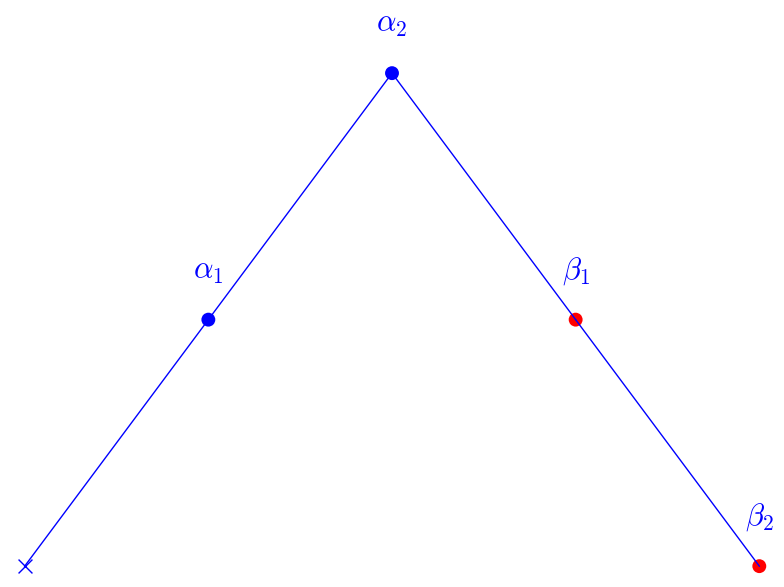}
      \subcaption{Diagram for zones 1 and 2}
      \label{plot}
    \end{subfigure}
    \begin{subfigure}[c]{.45\linewidth}
      \centering
      \includegraphics[width=.9\linewidth]{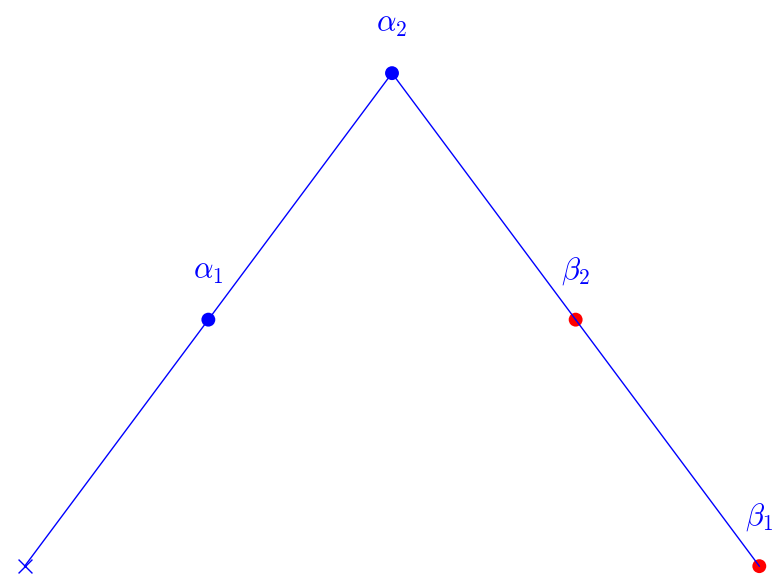}
      \subcaption{Diagram for zones 4 and 5}
      \label{zones}
    \end{subfigure}
  \end{figure}
\end{remark}

\noindent Remark that $[\gamma]$ is $0$ in zones $1, 4$, and $1$ in
zones $2,5$.  In the following table, we give a relation binding
$\lambda_1, r, x$ obtained by linear regression. The other column is
the formula for the parabolic degree in the given zone.

\begin{figure}[h]
  \centering
  \begin{tabular}{|c|c|c|}
    \hline
    Zone & $\lambda_1$ & $\deg_{par} H^{1,0}$ \\ \hline \hline
    1    & $2(1-2r)$   & $-1 + \{\gamma\} + \alpha_1 + 1 - \beta_2$\\ \hline
    2    & $2(r-x)$    & $\alpha_1 + 1 -\beta_2$\\ \hline
    3    & $0$         & $0$ \\ \hline
    4    & $2(x-2r)$   & $-1  + \{\gamma\} + \alpha_1 + 1-\beta_1$\\ \hline
    5    & $2r$        & $\alpha_1 + 1-\beta_1$ \\ \hline
  \end{tabular}
\end{figure}

\noindent In this case, the VHS is of weight $\leq 1$ and thus is in
the setting of \cite{Kontsevich}. In consequence, we have the equality
$$ \lambda_1 = 2 \frac {\deg_{par}{\mathcal E^{1}}} {\chi (S)}$$

\noindent Where $\deg_{par}$ is the parabolic degree of the
holomorphic bundle and $\chi(S) = 1$ the Euler characteristic of $S$.\\

\noindent This is a good test for our algorithm and formula on
degree. More generally, for any dimension $n$, this formula will hold
as long as the weight is equal to $1$.

\subsection{A peep to weight $2$}

Let $n$ be equal to $3$. In this case, there will be three Lyapunov
exponents $\lambda_1, 0, -\lambda_1$.  As explained in the previous
subsection, if the weight of the VHS is $0$, $\lambda_1 = 0$; if it is
$1$, $\lambda_1$ is equal to twice the parabolic degree of
$\mathcal E^1$.  We consider configurations where the weight is $2$.
Assume $\alpha_1 = 0$, the only cyclic order in which the VHS is
irreducible and of weight $2$ is for,
$$0 = \alpha_1 < \alpha_2 < \alpha_3 < \beta_1 < \beta_2
< \beta_3 < 1$$

\begin{figure}[h]
  \begin{subfigure}[l]{.4\linewidth}
    \includegraphics[width=\linewidth]{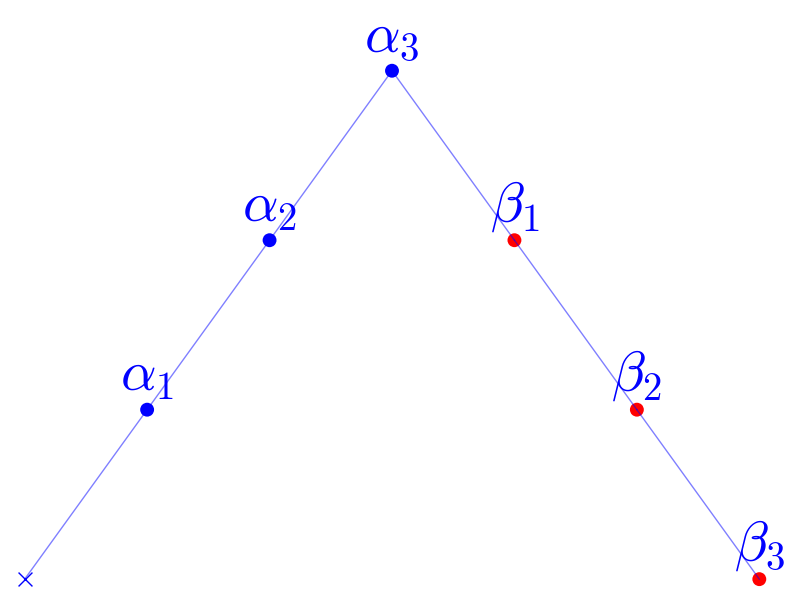}
    \hspace{2cm}
  \end{subfigure}
  \begin{subfigure}[r]{.4\linewidth}
    \hspace{1cm}
    \includegraphics[width=\linewidth]{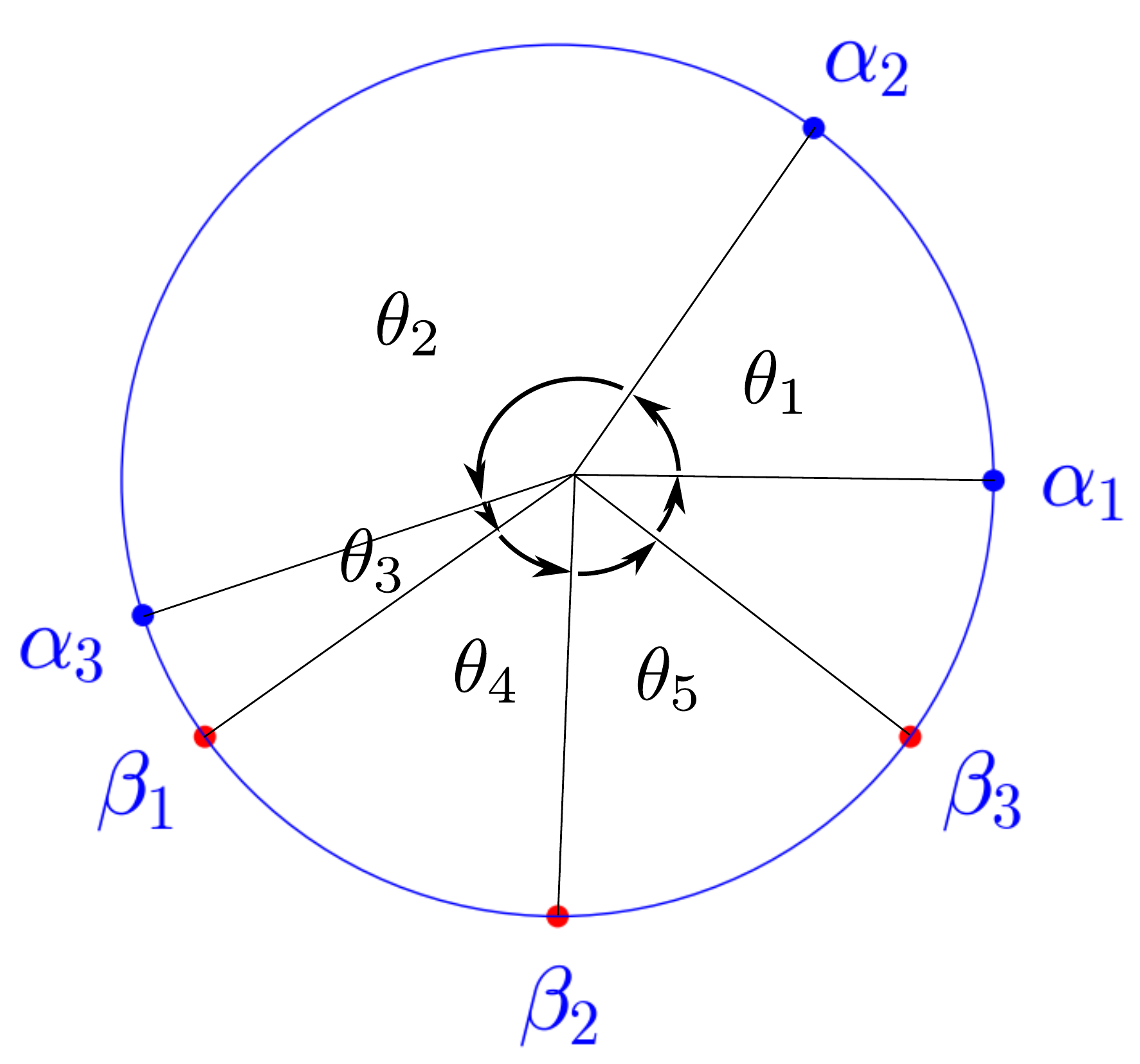}
  \end{subfigure}
\end{figure}

\noindent We parametrize these configurations with $5$ parameters
which will correspond to the distance between two consecutive
eigenvalues :
$\theta_1 = \alpha_2 - \alpha_1, \theta_2 = \alpha_3 - \alpha_2,
\theta_3 = \beta_1 - \alpha_3, \theta_4 = \beta_2 - \beta_1, \theta_5
= \beta_3 - \beta_2$.\\

\noindent Using a Monte-Carlo process, we found some values in this
configuration for which there is equality with twice the parabolic
degree of $\mathcal E^{2} \oplus \mathcal E^1$. We remarked that
several parameter points where there is equality satisfy
$\theta_1 = \theta_2$ and $\theta_4=\theta_5$.  This motivated us to
consider the $2$ dimensional subspace of parameters
$$(\theta_1, \theta_2, \theta_3, \theta_4, \theta_5) = (x, x, 1/2, y, y)$$
For these parameters we can observe a remarkable phenomenon; the
difference between the Lyapunov exponent and the formula with
parabolic degrees depends only on $x+y$. We plot this difference in
the Figure below and see that for some values of $x+y$ there is
equality.

\begin{figure}[h]
  \begin{subfigure}[l]{.4\linewidth}
    \includegraphics[width=\linewidth]{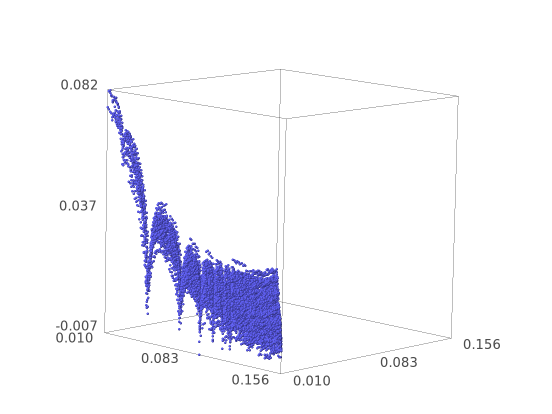}
    \subcaption{side}
  \end{subfigure}
  \begin{subfigure}[r]{.4\linewidth}
    \includegraphics[width=\linewidth]{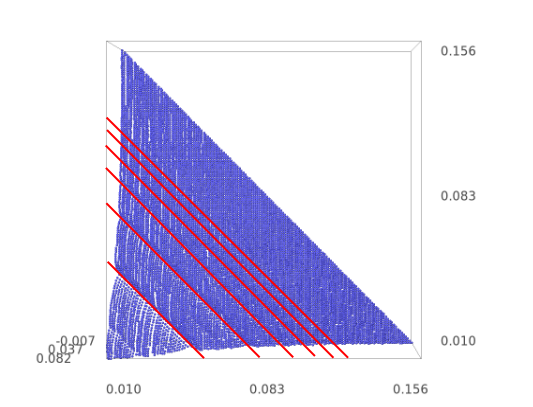}
    \subcaption{top}
  \end{subfigure}
\end{figure}

We computed that for $x+y = 25/3, 50/9 \text{ or } 1/10$ the formula
holds.

\newpage
\bibliographystyle{alpha} \bibliography{biblio}
\end{document}